\documentclass[12pt]{extarticle}
\usepackage{amsmath, amsthm, amssymb, hyperref, color}
\usepackage{graphicx}
\usepackage{caption}
\usepackage{subcaption}
\usepackage{mathtools}
\usepackage{enumerate}
\usepackage[all]{xypic}
\usepackage{verbatim}
\usepackage{upgreek}
\usepackage{array}
\usepackage{chemfig, chemnum}
\usepackage{tikz}
\usepackage{slashbox}
\usepackage[lined,commentsnumbered,ruled]{algorithm2e}
\usepackage{caption}
\usepackage{ifthen}

\oddsidemargin \evensidemargin
\theoremstyle{theorem}
\newtheorem{theorem}{Theorem}
\newtheorem{proposition}[theorem]{Proposition}

\theoremstyle{definition}
\newtheorem{definition}[theorem]{Definition}

\newtheorem{remark}[theorem]{Remark}

\newtheorem{examplex}[theorem]{Example}
\newenvironment{example}
  {\pushQED{\qed}\examplex}
  {\popQED\endexamplex}

\newcommand{\PP}{\mathbb{P}}
\newcommand{\RR}{\mathbb{R}}

\newcommand{\CC}{\mathbb{C} }
\newcommand{\ZZ}{\mathbb{Z}}
\newcommand{\NN}{\mathbb{N}}

\title{Likelihood Equations and Scattering Amplitudes}

\author{Bernd Sturmfels and Simon Telen}

\date{}

\begin{document}

\maketitle
\begin{abstract}
\noindent 
We relate scattering amplitudes in particle physics 
to maximum likelihood estimation for discrete models in algebraic statistics.
The scattering potential plays the role of the log-likelihood function, and its
critical points are solutions to rational function equations.
We study the ML degree of low-rank tensor models in statistics,
and we revisit physical theories proposed by Arkani-Hamed, Cachazo and their collaborators.
Recent advances in numerical algebraic geometry 
are employed to compute and certify critical points.
We also discuss positive models and how to compute their
string amplitudes.
\end{abstract}

\section{Introduction} \label{sec:introduction}

Likelihood equations are equations among
rational functions that arise in various contexts,
notably in high energy physics \cite{arkani2017positive, NAH} and in algebraic statistics \cite{ASCB, Sul}.
We establish a new link between these two fields.
This is interesting for both sides, and may lead to unexpected
advances in nonlinear algebra \cite{MS}.
Specifically, we develop the connection between
maximum likelihood estimation \cite{CHKS, HS} and the
geometric theory of scattering amplitudes~\cite{AG, cachazo2014scattering}.
Our goal is the practical solution of likelihood equations
with certified numerical methods~\cite{BRT, BT}.

On the statistics side, a discrete model is a subvariety
$X$ of the real projective space~$\PP^n$, which
is assumed to intersect the simplex $\Delta_n$ of
positive points. The homogeneous coordinates 
$p = (p_0:p_1:\cdots:p_n)$
are interpreted as unknown probabilities 
for the $n+1$ states, subject to the
constraint that $p$ lies in the model $ X$.
When collecting data, we write
$s_i$ for the number of times the $i$th state was observed.
The data vector  $s = (s_0,s_1,\ldots,s_n )$ is
also viewed modulo scaling, i.e. $s$ lies in
$\Delta_n \subset \PP^n$. We are interested in the {\em log-likelihood function}
\begin{equation}
\label{eq:LLF}
s_0 \cdot {\rm log}(p_0) \,+ \,
s_1 \cdot {\rm log}(p_1) \,+ \,\cdots + \,s_n \cdot {\rm log}(p_n) \,- \,
(s_0{+}s_1{+}\cdots{+}s_n) \cdot {\rm log}(p_0{+}p_1{+} \cdots {+} p_n).
\end{equation}
This is a well-defined function on $\Delta_n \subset \PP^n $. The aim
of likelihood inference in data analysis is to maximize (\ref{eq:LLF})
over all points $p$ in the model $X \cap \Delta_n$. In
algebraic statistics, we care about all complex critical points.
Their number, for generic $s$, is the {\em maximum likelihood  (ML) degree}
of the model $X$.  If $X$ is smooth then
the ML degree equals the Euler characteristic of the open variety
$X^o$, which is the complement of the
divisor in $X$ defined by  $p_0 p_1 \cdots p_n (\sum_{i=0}^n p_i) = 0$.
Computing ML degrees and identifying critical points is an active area of research \cite{RW}.

The situation is similar in the study of potentials and
associated amplitudes in physics.
Here the role of the data vector $s$ is played by the
vector of {\em Mandelstam invariants}, which is
constrained to lie in the {\em kinematic space}.
This mirrors the constraint that the coefficients 
in (\ref{eq:LLF}) sum to zero. In recent physical theories \cite{arkani2018scattering, AG, cachazo2014scattering},
the variety $X^o$ is the configuration space of $m$ points in
general position in $\PP^{k-1}$, up to projective transformations.
This is modeled by  the Grassmannian
${\rm Gr}(k,m) \subset \PP^{\binom{m}{k}-1}$, modulo the action of the torus $(\CC^*)^m$.
Let ${\rm Gr}(k,m)^o$ be the open Grassmannian where all
Pl\"ucker coordinates are nonzero.
We work in the $(k-1)(m-k-1)$-dimensional
manifold $X^o = {\rm Gr}(k,m)^o/(\CC^*)^m$.
The ML degree of $X^o$ is the number of
critical points on~$X^o$ of the potential function,
for generic $s$. A {\em scattering amplitude} is
the sum of a certain rational function over all critical points.
This is a  global residue \cite{CD}, so it evaluates to a
rational function in the Mandelstam invariants~$s$.

The present article is organized as follows.
Section~\ref{sec2} develops the promised connection for the
Grassmannian of lines $(k=2)$. Here $X^o$ is the
moduli space $\mathcal{M}_{0,m}$ of  $m$ marked points in $\PP^1$.
This prominent space is here recast as a statistical model. The ML degree of that model is
$(m-3)!$ and all critical points are real, thanks to Varchenko's Theorem \cite[Theorem 1.5]{ASCB}.
Our computational results for $m \leq 13$ are found in Table \ref{tab:scatteringk2}.
This is extended to arbitrary linear statistical models in Section~\ref{sec3}.
We show in that setting how  the software {\tt HomotopyContinuation.jl}  \cite{BRT, BT}
is used to find and certify all critical points of (\ref{eq:LLF}).
A key idea is to refrain from clearing denominators and
work with rational functions directly.

Section \ref{sec4} concerns higher Grassmannians $(k \geq 3)$ and their
associated likelihood equations.   We focus on the 
case $m=8, k=3$, where the amplitudes literature \cite{CUZ, AG}
reports the ML degree $188112$. We interpret this CEGM theory as a nonlinear
statistical model with $n=47$. Our method
computes and certifies all $188112$ critical points in a few minutes,
for random Mandelstam invariants with $s \geq 0$ in (\ref{eq:LLF}),
and we show that most of them are real.

In Section \ref{sec5} we apply our approach to a
class of models that is important in statistics, namely
conditional independence of identically distributed 
random variables. This corresponds to
symmetric tensors of low rank, so here $X$ is a 
Veronese secant variety. We determine the
ML degree in several new cases, well beyond the degree
 $12$  for tossing coins in the running example of \cite{HKS}.
This opens up a new chapter
in likelihood inference for tensors.

In Section \ref{sec6} we finally turn to amplitudes.
We build on the theory of stringy canonical forms due to
Arkani-Hamed, He and Lam \cite{arkani2019stringy}.
Definition \ref{def:positive} introduces a 
statistical version~of positive geometries
\cite{arkani2018scattering, arkani2020positive}.
 The string amplitudes in
 \cite{arkani2019stringy} are limits of their
marginal likelihood integrals.
They can be computed combinatorially from Newton polytopes, or
as global residues, by summing the reciprocal toric Hessian of 
the function (\ref{eq:LLF}) over its critical points.

\section{Points on the Line} \label{sec2}

We begin with a first direct connection between algebraic statistics and particle physics.
 The $m$-particle CHY scattering equations   \cite{cachazo2014scattering}  will be presented
  as likelihood equations for a linear statistical model on the moduli space $\mathcal{M}_{0,m}$.
  We introduce these rational function equations, and we solve them using 
  state-of-the-art tools from numerical algebraic geometry~\cite{BRT, BT, sommese2005numerical}.

We consider $m \geq 4$ points in $\PP^1$ whose homogeneous coordinates are the columns of
 \begin{equation} \label{eq:overparam}
 \begin{bmatrix}
\phantom{-} 0 & 1 &  1    & 1 & \cdots & 1  & 1  & 1 \,\,\\
 -1 & 0 & x_1& x_2 & \cdots & x_{m-4} & x_{m-3} & 1  \,\,
 \end{bmatrix}.
\end{equation}
We write $q_{ij}$ for the $2 {\times} 2$ minor given by the $i$-th and the $j$-th column of this 
$2 {\times} m$-matrix. The moduli space $\mathcal{M}_{0,m} =  {\rm Gr}(2,m)^o \big/ (\CC^*)^m $
is the set of points for which these minors are non-zero.
This is the complement of a hyperplane arrangement in $\CC^{m-3}$.
The corresponding real arrangement in $\RR^{m-3}$ has $(m-3)! $ bounded regions,
given by the possible orderings of $x_1,x_2,\ldots,x_{m-3}$ in $[0,1]$.
These regions are simplices and they define a triangulation of the cube $[0,1]^{m-3}$.
One of them  is the  positive region $\,\mathcal{M}_{0,m}^+ = \{ 0 < x_1 < x_2 < \cdots < x_{m-3}< 1 \}$.

We now define a statistical model $X$ on $n+1 = m(m-3)/2$ states. The states are the pairs
$(i,j)$ where $2 \leq i < j \leq m$ and $(i,j) \not= (2,m)$. 
The parameter vector $(x_1,\ldots,x_{m-3})$ is assumed to lie in  $\mathcal{M}_{0,m}^+$.
The probability of observing 
the state $(i,j)$ is $\, p_{ij} \, = \,  \alpha_{ij} q_{ij}$, where
\begin{equation}
\label{eq:cleverchoice}
\alpha_{im} = \frac{1}{m-3}\, , \,\,
\alpha_{ij} = \frac{1}{(m-3)^2} \quad {\rm and} \quad
\alpha_{2j} = \frac{2m{-}2j{-}1}{(m-3)^2} \qquad {\rm for} \quad 3 \leq i < j \leq m-1. 
\end{equation}
These positive constants are chosen so that the
sum of the $n+1$ linear expressions $p_{ij}$ equals~$1$.
 
 Suppose we collect data. For each of the $n+1$ states $(i,j)$ as above, we record the
 number $s_{ij}$ of observations of that state. The aim of statistical inference
 is to find the point $\hat x = (\hat x_1, \hat x_2, \ldots, \hat x_{m-3} )$ in  the parameter space $ \mathcal{M}_{0,m}^+$
 that best explains the data. Adopting the classical frequentist framework, this is done by maximizing
 the log-likelihood function  \begin{equation} \label{eq:potential}
 L(x) \,\,\,=\,\,\, \sum_{(i,j)} s_{ij} \log (p_{ij}(x))\, \,\, = \,\,\, \sum_{(i,j)} s_{ij} \log(q_{ij}(x))\,+\, const.
 \end{equation}
 We write ${\rm Crit}(L)$ for the set of critical points of $L$, i.e.~the
solutions of the likelihood equations
\begin{equation} \label{eq:scatteringk2}
\frac{\partial L}{\partial x_1} \,=\, \frac{\partial L}{\partial x_2} \,=\, \,\cdots \,\,= \,\frac{\partial L}{\partial x_{m-3}} \,=\, 0.
\end{equation}
This is a system of $m-3$ rational function equations in the $m-3$ unknowns $x_1, \ldots, x_{m-3}$. 

\begin{proposition} \label{prop:varchenko}
If all $s_{ij}$ are positive  then  (\ref{eq:scatteringk2}) has precisely $(m-3)!$ complex solutions.
All solutions are real, and 
there is one solution for each of the orderings of the $m-3$ coordinates.
\end{proposition}

\begin{proof}
This result is known in the physics literature. We here derive it from
Varchenko's Theorem in Algebraic Statistics
\cite[Theorem~13]{CHKS}. This
states that the likelihood equations of a
 linear space $X$ have only real solutions, and there is one  solution
  in each bounded region of~the arrangement in 
  the real  linear space $X_\RR $ defined by the  $n+1$ hyperplanes $\{p_i = 0\}$.
  For the CHY model, we identify $X_\RR$ with the parameter space $\RR^{m-3}$, where
 the hyperplanes are $\{x_i = 0\}$, $\{x_i = 1\}$ and $\{x_i = x_j\}$. Every point
with $x_i < 0$ or $x_i > 1 $ for some $i$ can be moved to infinity without
crossing a hyperplane. This implies that the
 bounded regions are the simplices $\{0 < x_{\pi_1}  < \cdots < x_{\pi_{m-3}} < 1\}$,
where $\pi$ runs over all $(m-3)!$ permutations. \end{proof}

\begin{example}[$m=6,n=8$] \label{ex:sixeight}
We consider a linear model $X$ on nine states
 $23, 24, \ldots, 56$. Their probabilities, which sum to $1$, 
 are linear functions of three model parameters $x_1,x_2,x_3$:
 \begin{equation}
\label{eq:model62} \begin{matrix}
  p_{23} \,=\, 5 x_1 /9, &
  p_{24} \,=\, x_2/3, &
  p_{25} \,=\, x_3/9 , \\
  p_{34} = (x_2-x_1)/9, &
  p_{35} = (x_3-x_1)/9 ,&
  p_{45} = (x_3-x_2)/9,  \\
  p_{36} = (1-x_1)/3, &
  p_{46} = (1-x_2)/3, &
  p_{56} = (1-x_3)/3.
  \end{matrix}
 \end{equation}
 This maps
     the tetrahedron
  $\mathcal{M}_{0,6}^+ = \{\,0 < x_1 < x_2 < x_3 < 1\,\}$ into the~probability simplex $\Delta_8$.
Suppose we  collect data with sample size $170$, and the resulting data vector has coordinates
\begin{equation}
\label{eq:data62} \begin{small} s_{23} = 25,  \, s_{24} = 23,\, s_{25} = 16,
\, s_{34} = 12,\, s_{35} = 22,
\, s_{45} = 16,\,s_{36} = 14,\, s_{46} = 15,\, s_{56} = 27. \, \end{small}
\end{equation}
We must solve an optimization problem on
 $\mathcal{M}_{0,6}^+$, namely to maximize  the function
\begin{equation} 
\label{eq:loglike62} \begin{matrix} L \,\,\, = & \!\!\!\!\!\!\!\!\!\!
s_{23} {\rm log}(p_{23})+
s_{24} {\rm log}(p_{24})+
s_{25} {\rm log}(p_{25})+
s_{34} {\rm log}(p_{34})\,+ \qquad \qquad \\ & \,\,\quad
s_{35} {\rm log}(p_{35})+
s_{45} {\rm log}(p_{45})+
s_{36} {\rm log}(p_{36})+
s_{46} {\rm log}(p_{46})+
s_{56} {\rm log}(p_{56}).
\end{matrix}
\end{equation}
The set ${\rm Crit}(L)$ has one point in
each bounded region of
the arrangement of nine planes $\{ p_{ij} = 0 \}$ in~$\RR^3$.
The six bounded regions lie in the
cube $[0,1]^3$. They correspond to the orderings of the values $x_1, x_2,  x_3$.
For instance, for the data in (\ref{eq:data62}), the six critical  points are 
$$ \begin{small} \begin{matrix}
\hat x_1 = 0.240043275929170 , & \hat x_2 = 0.508172206739870, & \hat x_3 = 0.777005866817260; \\
  x_1 = 0.223437550855307 , &   x_2 = 0.843543048681696, &   x_3 = 0.518706389808326; \\
  x_1 = 0.481967726451097 , &   x_2 = 0.235545240880672, &   x_3 = 0.781115679885971; \\
  x_1 = 0.618277926209287 , &   x_2 = 0.851974456945199, &   x_3 = 0.155992558374125; \\
  x_1 = 0.861996060709608 , &   x_2 = 0.217605043343923, &   x_3 = 0.453238947004789; \\
  x_1 = 0.863192417250353 , &   x_2 = 0.578669456252017, &   x_3 = 0.157960116395912. \\
\end{matrix} \end{small}
$$
The first triple is the maximum likelihood estimate.
The learned distribution
 in the model~is
\begin{equation}
\label{eq:learned}  \begin{small} \begin{matrix} 
\hat p_{23} = 0.13336 , &
\hat p_{24} = 0.16939 , & 
\hat p_{25} = 0.08633, &
\hat p_{34} = 0.02979, &
\hat p_{35} = 0.05966 , \\
\hat p_{36} = 0.25332, & 
\hat p_{45} = 0.02987, & 
\hat p_{46} = 0.16394, &
\hat p_{56} = 0.07433. &
\end{matrix} \end{small}
\end{equation}
We shall see that  this computation can be done for much larger values of $m$ and~$n$.
\end{example}

 We now turn to physics. In quantum field theory, the  $s_{ij}$ are known as \emph{Mandelstam invariants}.
 One writes them in a symmetric $m \times m$-matrix with zeros on the diagonal, so we
 have $s_{ii} = 0$ and    $s_{ij} = s_{ji}$. Momentum conservation means that
    the row sums are zero, i.e.~$ \sum_{j=1}^m s_{ij} = 0\,$ for $ i = 1, \ldots, m$.
  These equations define the \emph{kinematic space}, which has dimension   $n+1 = \binom{m}{2} - m$.
  On that space,  the $m$ Mandelstam invariants $  s_{12}, s_{13}, \ldots, s_{1m}$ and
  $s_{2m}$ can be written uniquely in terms of our counts
  $s_{ij}$ in the statistical model above.
  For instance, for $m=6$, the kinematic space is parametrized by the nine counts
in (\ref{eq:data62}) via
\begin{equation}
\label{eq:kinemat}
 \begin{matrix}
s_{12} = s_{34} + s_{35} + s_{36} + s_{45} + s_{46} + s_{56},&
s_{13} = -s_{23} - s_{34} - s_{35} - s_{36},\, \\
s_{16} = s_{23} + s_{34} + s_{35} + s_{24} + s_{45} + s_{25},  &
s_{14} = -s_{24} - s_{34} - s_{45} - s_{46},    \\
s_{26} = -s_{23}{-}s_{34}{-}s_{35}{-}s_{36}{-}s_{24}{-}s_{45} {-}s_{46}{-}s_{25} {-} s_{56}, &
s_{15} = -s_{25} - s_{35} - s_{45} - s_{56}. 
\end{matrix}
\end{equation}
  The scattering potential in the CHY model 
  coincides with the log-likelihood function $L$,
  up to the additive constant in \eqref{eq:potential}. Hence
  the scattering equations are the likelihood equations.
      
\begin{table}[h]
\centering
\begin{tabular}{l|lllll}
$m$ & $ n+1 $ & $(m \! - \! 3)! $ & $t_\CC $ & $t_\RR$ & $t_{\text{cert}}$ \\ \hline
$10$ & \,\, 35 & 5040 & 0.75 & 0.28 & 0.5 \\
$11$ & \,\, 44& 40320 & 13.4 & 3.4 & 4.0 \\
$12$ & \,\, 54 & 362880 & 124.6 & 43.7 & 45.0 \\
$13$ & \,\, 65 & 3628800 & 2141.5 & 578.2 & 1178.0
\end{tabular}
\caption{Computing and certifying solutions to CHY scattering equations with the 
 method in Section \ref{sec3}. 
Here  $t_\CC, t_\RR, t_{\text{cert}}$ denote timings (in seconds)
that are explained  in Example~\ref{ex:computationsk2}.}
\label{tab:scatteringk2}
\end{table}

  We now come to the punchline of this section:
  {\em current off-the-shelf software from numerical algebraic geometry is
  highly efficient and reliable in solving our equations}. For our computations we
   used the {\tt julia} package   {\tt HomotopyContinuation.jl},
    due to Breiding and Timme \cite{BT}, including
  the recent certification feature \cite{BRT} which is based on interval arithmetic.

  In Table \ref{tab:scatteringk2} we present the timings we obtained
  for solving the scattering equations (\ref{eq:scatteringk2}) when
  the number of particles is    $m=10,11,12,13$.
  Recall that the solutions are the critical points of $L$ in the
  moduli space   $\mathcal{M}_{0,m}$.
    In later sections we apply these methods for solving likelihood equations coming from other statistical models,
    including higher Grassmannians.
    
The first two columns in  Table \ref{tab:scatteringk2}
show the number $n+1 = m(m-3)/2$ of states in the
statistical model and the ML degree $(m-3)!$.
The last three columns  show computation times.
The most relevant among these is $t_\RR$.
This is the time in seconds for computing 
all $(m-3)!$ real critical points for 
a given system of  Mandelstam invariants $s_{ij}>0$.
For instance, for $m=12$, it takes less than one
minute to compute all
$(12-3)!  = 362880$ solutions. 

 \section{Linear Models and How to Compute} \label{sec3}
 
We here explain our methodology for solving the likelihood equations.
For ease of illustration we consider linear statistical models,
with the understanding that the computations are
analogous for nonlinear models. 
The scope  of that becomes visible in the
next two sections.

Fix affine-linear polynomials $p_0(x),p_1(x),\ldots,p_n(x)$  with real coefficients
in $d$ unknowns $x = (x_1,x_2,\ldots,x_d)$. We assume that
$p_0(x) + p_1(x) + \cdots + p_n(x) = 1$ and
that the convex polytope 
$\Theta = \{\, x \in \RR^d \,: \,p_i(x) \geq 0 \,\}$ has dimension $d$.
The model is the $d$-dimensional linear space $X$ in $\PP^n$ 
parametrized by $x \mapsto (p_0(x): \cdots : p_n(x))$.
Given any positive real data vector $s = (s_0,s_1,\ldots,s_n)$,
we wish to find all critical points of the log-likelihood function 
$L$ in (\ref{eq:LLF}).  

By Varchenko's Theorem, all complex critical points are real,
and there is one critical point in each  bounded region of the
arrangement of $n+1$ hyperplanes $\{ p_i(x) = 0\}$ in $\RR^d$.
One of these bounded regions is the polytope $\Theta$, so this contains a unique
critical point $\hat x$.
 Its image $\hat p = p(\hat x)$ 
in $\Delta_n$ is the distribution in the model $X$ that best explains the data $s$.

The software  {\tt HomotopyContinuation.jl} \cite{BT} is
very user-friendly.
We will show how to compute all critical points with version 2.3.1.
We start by generating a random linear model:
\begin{verbatim}
@var x[1:d]												
c = rand(n+1); c = c/sum(c)
p = [randn(d)'*x + c[i] for i = 1:n]
p = push!(p,1-sum(p))
\end{verbatim}
The array {\tt p} contains $n+1$ affine polynomials in the unknowns {\tt x}.
Their constant terms are the positive reals in {\tt c} that sum to $1$. The polytope
  $\Theta$ has dimension $d$ since $0 \in {\rm int}(\Theta)$.    
   The next step is to construct the log-likelihood function and compute its derivatives. Using the
     logarithm function in {\tt HomotopyContinuation.jl}, this can be done in two lines of code:
\begin{verbatim}
@var s[0:n]
L = sum([s[i]*log(p[i]) for i = 1:n+1])
F = System(differentiate(L,x), parameters = s)
\end{verbatim}
Here {\tt F} represents the rational map
$F : \CC^d \times \CC^{n+1} \dashrightarrow \CC^d, \,(x,s) \mapsto
 \bigl( \frac{\partial L}{\partial x_1}, \ldots, \frac{\partial L}{\partial x_d } \bigr) $.
 We choose a 
  random complex data vector $s^* \in \CC^{n+1}$, and we solve the system $F(x;s^*) = 0$ as follows:
 \begin{verbatim}
monodromy_result = monodromy_solve(F)
s_star = parameters(monodromy_result)
\end{verbatim}
This uses the \emph{monodromy method} for solving a generic instance of a parametrized family~\cite{duff2019solving}. We stress that we do not turn rational functions into polynomials by clearing denominators. Working directly with the rational functions allows for cheaper evaluation of {\tt F} and it avoids spurious solutions in the hyperplanes 
$\{p_i(x) = 0 \}$. 
Once we have the solutions for $s^* \in \CC^{n+1}$, we can find the solutions for any data 
vector $(s^*)' \in \RR^{n+1}_{> 0}$ via a \emph{(straight line) coefficient parameter homotopy}.
Here the vector $s$ moves from $s^*$ to $(s^*)'$ along a straight line in $\CC^{n+1}$. 

As the \emph{start parameter values} $s^*$ move to the \emph{target parameter values} $(s^*)'$,
the solutions of $F(x;s^*) = 0$ move towards the solutions of $F(x;(s^*)') = 0$. We can track them numerically.
 For details, see \cite[Chapter 7]{sommese2005numerical}. The coefficient parameter homotopy is implemented in the {\tt solve} function. The following code solves 
  $F(x;(s^*)') = 0$ for random $(s^*)' \in  \RR^{n+1}_{>0}$:
\begin{verbatim}
startsols = solutions(monodromy_result)
s_star_prime = rand(length(s))
cp_result = solve(F, startsols; start_parameters = s_star,
                                target_parameters = s_star_prime)
\end{verbatim}
 The solutions computed via monodromy are stored in {\tt startsols}. These can be used as starting points in the coefficient parameter homotopy for solving any new instance $F(x;(s^*)') = 0$ of our
  equations.
  Hence, the monodromy computation happens only once for a given model.

Finally, we \emph{certify} the solutions found by the coefficient parameter homotopy using the certification technique 
described recently in \cite{BRT}. Each solution that has been certified is guaranteed to be an \emph{approximate solution},  in a suitable sense, to our system of equations. 
\begin{verbatim}
cert = certify(F, solutions(cp_result), s_star_prime) 
\end{verbatim}

In our discussion we described a workflow consisting of three steps:
monodromy, coefficient parameter homotopy, and certification.
These steps are easy to run, and they can be applied to any statistical model
and hence to any system of scattering equations in physics.
A nice feature of linear models, like CHY in Section \ref{sec2}, is that the method can solve
the likelihood equations using \emph{real arithmetic only}. This allows us to reduce the computation~time.

We now explain the real arithmetic idea. In general,
one uses complex start values $s^*$ 
is to avoid the discriminant locus of the family $F(x;s) = 0$. For linear models, this locus arises from the
{\em entropic discriminant}, which is a sum of squares by \cite[Theorem 6.2]{KV}.
The real locus has codimension $\geq 2$ and is disjoint from 
$\RR^{n+1}_{>0}$. As the data vector $s$ varies continuously in $\RR^{n+1}_{>0}$, the
solutions to   $F(x;s)= 0$ move in distinct bounded regions  in $\RR^d$.
 Therefore, once we have solved $F(x;s^*) = 0$ for some $s^* \in \RR^{n+1}_{>0}$, we can solve $F(x;(s^*)') = 0$ for
 any $(s^*)' \in \RR^{n+1}_{>0}$ via a straight line coefficient parameter homotopy that uses only real arithmetic. 
 In particular, for computing the MLE, we only need to track one solution, namely that in~$\Theta$.

\begin{example}[Scattering equations on $\mathcal{M}_{0,m}$] \label{ex:computationsk2}
Section~\ref{sec2} addressed a linear model from physics \cite{AG, cachazo2014scattering}
with $d=m-3$, $n= m(m-3)/2-1$ and ML degree $(m-3)!$.
Our computations for Table  \ref{tab:scatteringk2}
used the workflow described above.  The columns 
 $t_\CC$ and $t_{\textup{cert}}$ show the computation times (in seconds) for the coefficient parameter homotopy from $s^* \in \CC^{n+1}$ to $(s^*)' \in \RR^{n+1}_{>0}$ and for the certification respectively. The column  $t_\RR$ shows the 
  time for path tracking over the reals, from $s^* \in \RR^{n+1}_{>0}$ to $(s^*)' \in \RR^{n+1}_{>0}$. 
   In each run, all
    $(m-3)!$ solutions were certified. The time for the monodromy step is not reported, as it 
    is an \emph{off-line step} which happens only once. For instance, for $m=12$,
    the off-line  step  takes about 14 minutes. All computations were run on a 16 GB MacBook Pro
    with an Intel Core i7 processor working at 2.6 GHz.
\end{example}

\begin{example}[Random linear models] \label{ex:RLM}
We examined random models for various $(n,d)$.
Unlike in Section~\ref{sec2}, the $p_i(x)$ are now dense.
The number of bounded regions equals $\binom{n}{d}$.
This is the ML degree; see \cite[eqn~(8)]{HKS}.
 Using the same computer as in Example \ref{ex:computationsk2}, we obtained
 the results in Table \ref{tab:randomlinear}, for various  central binomial coefficients.
 Again, we do not report the timings for the off-line step, 
 which happens once per pair $(n,d)$. All 
 models in Table \ref{tab:randomlinear} were solved easily using the default settings in {\tt HomotopyContinuation.jl}.
 Larger values of $(n,d)$ are more challenging.
 The  straightforward approach we presented above
 ran into numerical difficulties. The monodromy loop
 sometimes failed to find a full set of starting solutions,
 and a few paths got lost in  the coefficient parameter homotopy.
 Solving larger problems reliably will require a more clever approach
   or more conservative settings. \end{example}

\begin{table}[h]
\centering
\begin{tabular}{l|llll}
$(n,d)$ & $ \binom{n}{d} $ & $t_\CC $ & $t_\RR$ & $t_{\text{cert}}$ \smallskip \\ \hline
$(12,6)$ & 924 & 0.25 & 0.09 & 0.15 \\
$(13,6)$ & 1716 & 0.46 & 0.13 & 0.27 \\
$(14,7)$ & 3432  & 1.34 & 0.44  & 0.87 \\
$(15,7)$ & 6435  & 2.12 & 0.87  & 1.46 \\
$(16,8)$ &12870 & 5.06 & 2.00  & 2.91 \\
$(17,8)$ &24310 & 6.25 & 3.60  & 7.25 \\
\end{tabular}
\caption{Solving the likelihood equations for random linear models.
Here $\binom{n}{d}$ is the ML degree, 
 $t_\CC$ and $ t_\RR $ are the timings for solving,
 and $t_{\text{cert}}$ is the timing for certifying the solutions.}
 \label{tab:randomlinear}
\end{table}

Examples \ref{ex:computationsk2} and \ref{ex:RLM} lead to the following conclusion.
The special combinatorial structure of the CHY scattering equations allows 
  us to solve large instances with a fairly naive method.
 Things are different for generic linear models.
 We encountered numerical issues for the default settings when the ML degree 
 exceeds $20000$.
 The same dichotomy occurs for the models studied in the next
 two sections. Low degree and sparsity 
 render the equations from physics especially
  suitable for reliable and certified computations
with {\tt HomotopyContinuation.jl}.

\section{Higher Grassmannians} \label{sec4}

Let ${\rm Gr}(k,m)$ denote the Grassmannian in its Pl\"ucker embedding in
$\PP^{\binom{m}{k}-1}$, with Pl\"ucker coordinates $p_I$ indexed by
increasing sequences $I = (1 \leq i_1 < i_2 < \cdots < i_k \leq m)$.
We~write ${\rm Gr}(k,m)^o$ for the open part
where all $p_I$ are nonzero and
 $X^o$ for its quotient modulo $(\CC^*)^m$.
We represent each point in $X^o$ by a $k \times m$ matrix 
that has been normalized and contains $(k-1)(m-k-1) = {\rm dim}(X^o)$ unknowns.
There are different conventions for setting this~up.
For $k=3$, we place $2m-8$ unknowns $x_1,\ldots,x_{m-4}$
and $y_1 ,  \ldots, y_{m-4}$ in the matrix as follows:
\begin{equation}
\label{eq:threebym}
 \begin{bmatrix} 
\,\, 0 & 0 & 1 & 1  & 1 & 1 & 1 & \cdots & 1 \,\,\\
\,\, 0 & -1 & 0 & 1 & x_1 & x_2 & x_3 & \cdots & x_{m-4}  \,\,\\
\,\, 1 & 0 & 0 & 1 & y_1 & y_2 & y_3 & \cdots & y_{m-4} \,\,
\end{bmatrix}.
\end{equation}
This ensures that $m$ special minors $p_I$ are equal to $1$.
These are the minors indexed by
\begin{equation}
\label{eq:excluded}
I \,\,=\,\, 
123, 124, \ldots, 12m ,\, 134, 234.
\end{equation}
The following result is known in the literature on scattering amplitudes; see
 \cite[Section~7.1]{arkani2019stringy},  \cite[Section 3]{CUZ} and \cite[Appendix C]{AG}.
Our  computations  furnish  an independent~verification.

\begin{proposition} \label{prop:3degrees}
The ML degree of the models $X^o$ for $n=6,7,8$ equals
$26$, $1272$ and $188112$.
\end{proposition}

\begin{proof}[Sketch of Proof]
The certification with
 {\tt HomotopyContinuation.jl}
 furnishes a solid proof of the lower bound. 
 The proof is an identity in
 interval arithmetic \cite{BRT}.
The upper bound requires more work.
We can either use the degenerations known as soft limits
\cite{CUZ}, or Thomas Lam's approach (mentioned in \cite[Section 1]{CUZ})
that rests on finite fields and the Weil conjectures, or the trace test
method in numerical algebraic geometry. It would  be desirable
to find a general formula and theoretical understanding for the Euler characteristic 
of $X^o = {\rm Gr}(k,m)^o$.
\end{proof}

In  the development of algebraic statistics there was an earlier attempt to view the Grassmannian ${\rm Gr}(k,m)$
as a discrete statistical model. It has dimension $k(m-k)$, it has
$n+1=\binom{m}{k}$ states, and the Pl\"ucker coordinates are the probabilities.
 We refer to \cite[Section~5]{HKS} where the numbers $4$ and $22$
were reported for the ML degrees of ${\rm Gr}(2,4)$ and ${\rm Gr}(2,5)$.
That model is different from the one studied here, where the dimension is
$(k-1)(m-k-1)$, the number of states is $n+1 = \binom{m}{k} - m$,
and ${\rm Gr}(2,m)$ has ML degree $(m-3)!$.
In light of the ubiquity and importance of the moduli space $\mathcal{M}_{0,m}$,
we have concluded that the physical model $X^o$ above is the better
way to think about the Grassmannian in the setting of algebraic statistics.

In what follows we work in the set-up for $k=3$ as in  \cite{CUZ, AG}.
The task is to compute the set ${\rm Crit}(L)$ of critical points of the scattering potential $L = \sum_I s_I {\rm log}(p_I)$.
We assign positive reals to the $\binom{m}{3}-m$ 
 Mandelstam invariants $s_I$ where $I$ is any triple
not listed in (\ref{eq:excluded}). 
 The $m$ remaining Mandelstam invariants $s_I$ from  (\ref{eq:excluded})  
are determined from the kinematic relations
$$  \sum_{jk} s_{ijk} \,=\, 0, \qquad {\rm for} \quad i = 1, \ldots, m.$$
Here $(s_{ijk})$ is a symmetric tensor with $s_{ijk} = 0$
unless $i,j,k$ are distinct; see \cite[eqn~(1.6)]{AG}.
  Rewriting the kinematic equations, we
obtain formulas that are analogous to (\ref{eq:kinemat}).
However, the $m$ Mandelstam invariants $s_I$ from (\ref{eq:excluded})
do not matter for us, since ${\rm log}(p_I) = 0$, so
they do not appear in the scattering potential  $L$.
For the other $n+1$ indices $I$, the polynomials $p_I$
are bilinear in the unknowns $x_i,y_i$. In conclusion,
our task is to solve a system of $2m-8$ rational function equations
in $2m-8$ unknowns, namely
$\,\frac{\partial L}{\partial x_i} =  \frac{\partial L}{\partial y_i} = 0$ for $i=1,\ldots,m-4$.

\smallskip

We use the techniques from Section \ref{sec3} to solve these equations for $m = 6, 7, 8$. The results are reported in Table \ref{tab:scattering} using the same notation as in the previous sections.
\begin{table}[]
\centering
\begin{tabular}{l|lllll}
$m$ & $n + 1 $ & ML degree & $t_\CC$ & $t_{\text{cert}}$ \\ \hline
$6$ & 14 & 26 & 0.02 & 0.01 \\
$7$ & 28 & 1272  & 0.35 &  0.19 \\
$8$ & 48 & 188112  & 70.03 & 47.71
\end{tabular}
\caption{Computation times for solving the CEGM scattering equations.}
\label{tab:scattering}
\end{table}
For $m = 6$, we confirmed that all $26$ solutions are real (cf.~\cite[Appendix C]{AG}).
In the case $m = 7$,  all $1272$ solutions are computed in a fraction of a second. 
For concreteness, let us consider the data
\begin{equation}
\label{eq:data37} \!\! \begin{small} \begin{matrix}
s_{135}=45, & s_{235}=597, & s_{145}=473, & s_{245}=745 , & 
s_{345}=29 ,& s_{136}=296, & s_{236}=503, &  \\ s_{146}=725, & s_{246}=402, & 
s_{346}=132, & s_{156}=557, & s_{256}=649, & s_{356}=461, & s_{456}=246, & 
 \\ s_{137}=636, & s_{237}=662, & s_{147}=37, & s_{247}=945, & s_{347}=87, & 
 s_{157}=613 , & s_{257}=819, &  \\ s_{357}=889, & s_{457}=473, & 
 s_{167}=665, & s_{267}=57, & s_{367}=340 , & s_{467}=621, & s_{567}=562.
 \end{matrix} \end{small} 
 \end{equation}
 These are the $n+1 = 28$ Mandelstam invariants not in (\ref{eq:excluded}).
 For these data, we found $1272$ solutions in 0.35 seconds, and we certified them
 in 0.19 seconds. Precisely $1210$ of the solutions are real.
   To the best of our knowledge, no complete set of solutions to the scattering
   equations for $m = 7$  with general $s_{ijk}$ has  been reported in the literature so far.
  
 Using {\tt HomotopyContinuation.jl} we can also solve the likelihood equations for $m = 8$.
 This works in the order of minutes. But there are challenges for this large 
 nonlinear model. While our earlier models showed the power of {\tt solve} as a blackbox 
 routine, here the situation is more delicate.
   It may happen that not all $ 188112$ paths are tracked successfully in the coefficient parameter homotopy. 
   For an example, fix the $n+1 = 48$    Mandelstam invariants 
   \begin{equation}
\label{eq:bigmandelstam} \begin{small} \begin{matrix}
s_{135}=632 , & s_{235}=5076 , & s_{145}=6368 , & s_{245}=619 , & s_{345}=8083 , & s_{136}=5762 , &  \\ s_{236}=2099 , & s_{146}=7767 , & s_{246}=9208 , & s_{346}=4889 , & s_{156}=4412 , & s_{256}=1024 , &  \\ s_{356}=5988 , & s_{456}=924 , & s_{137}=3430 , & s_{237}=1017 , & s_{147}=6235 , & s_{247}=8010 , &  \\ s_{347}=9867 , & s_{157}=2364 , & s_{257}=9661 , & s_{357}=7008 , & s_{457}=4706 , & s_{167}=2892 , &  \\ s_{267}=7670 , & s_{367}=5769 , & s_{467}=3188 , & s_{567}=9696 , & s_{138}=6264 , & s_{238}=5878 , &  \\ s_{148}=1442 , & s_{248}=1501 , & s_{348}=4225 , & s_{158}=579 , & s_{258}=7524 , & s_{358}=394 , &  \\ s_{458}=878 , & s_{168}=7684 , & s_{268}=5985 , & s_{368}=9306 , & s_{468}=8429 , & s_{568}=648 , &  \\ s_{178}=697 , & s_{278}=8414 , & s_{378}=3151 , & s_{478}=369 , & s_{578}=3176 ,& s_{678}=8649 .
\end{matrix}
\end{small} 
\end{equation}

Starting with the output {\tt startsols} from the off-line phase,
the command {\tt solve} finds
   $188109$ distinct solutions in $70$ seconds.
   The remaining three solutions are found by a few extra
         minutes of monodromy loops. The 
         $188109$ earlier solutions in {\tt cp\_result} serve as seeds:          
\begin{verbatim}
R = monodromy_solve(F,solutions(cp_result),s_star_prime) 
\end{verbatim}

When running the off-line step for any new statistical model, it is
very helpful to know the ML degree ahead of time. In our situation,
with knowledge of Proposition \ref{prop:3degrees},
we can use the option {\tt  target\_solutions\_count = 188112}  in the command
{\tt monodromy\_solve},  both for off-line and for on-line.
This interrupts the monodromy  loop when all solutions are found. 

All in all, the on-line phase for a given vector of Mandelstam invariants
takes no more than a few minutes. This includes the coefficient parameter homotopy, the on-line monodromy phase
described above, and the certification step that furnishes the proof of correctness.

\begin{remark}
A notable feature of the $k=2$ model in Section~\ref{sec2} is
that all critical points of the log-likelihood function are real (Proposition \ref{prop:varchenko}).
This is no longer true for $k \geq 3$. However, we observed
experimentally that most of  the solutions are real.
In particular, for the data in (\ref{eq:bigmandelstam}),
precisely $149408$ out of $188112 $ critical points are real. 
We do not know whether the $48$ Mandelstam invariants
$s_{ijk}$ can be chosen so that all $188112 $ complex 
solutions are  real. 
\end{remark}

\begin{remark} It would be interesting to
investigate the likelihood geometry of positroid cells in ${\rm Gr}(k,m)$,
taken modulo the $ (\CC^*)^m$ action as in \cite{arkani2020positive}.
The software {\tt HomotopyContinuation.jl} will be useful for
finding the ML degrees of such models.
For these computations, one replaces the matrices in
(\ref{eq:overparam}) and
(\ref{eq:threebym}) with the network parametrization
 of positroid cells \cite{arkani2020positive, TW}.
\end{remark}

\section{Low Rank Tensors} \label{sec5}

In this section we return to algebraic statistics.
We apply our methods to the model of
{\em conditional independence for identically distributed 
random variables}. This corresponds to
symmetric tensors of low rank. We here study their
ML degree and likelihood equations.

We consider symmetric tensors of format $m \times m \times \cdots \times m$ 
where the number of factors is~$\ell$. Our model $X$ is
the variety of symmetric tensors of rank $\leq k$, or equivalently,
the $k$th secant variety of the $\ell$th Veronese embedding of $\PP^{m-1}$.
The dimension of the model equals  ${\rm dim}(X) = km -1$.
We follow the set-up in (\ref{eq:LLF}), with the number of states  $n+1 = \binom{m+\ell-1}{\ell}$.  The state
space is the set $\Omega_{m,\ell}$ of  sequences $I = (i_1,i_2,\ldots,i_m) \in \NN^m$
 with $i_1 + i_2 + \cdots + i_m = \ell$.

The parameter space for our statistical model is the polytope
$\Theta = (\Delta_{m-1})^k \times \Delta_{k-1}$,
where the points $x_i$ in the $i$th simplex $\Delta_{m-1}$ 
are distributions on  the $i$th random variable with $m$ states, and points $y$
in the simplex $\Delta_{k-1}$ specify the mixture parameters.
Hence $x = (x_{i,j})$ is a nonnegative $k \times m$ matrix
whose rows sum to $1$, and $y$ is a nonnegative vector in
$\RR^k$ whose entries sum to $1$. The probability of observing the state $I = (i_1,i_2,\ldots,i_m)$ equals
\begin{equation}
\label{eq:tensorpara}
 p_I(x,y) \,\, = \,\,\frac{\ell!}{i_1! i_2! \cdots i_m!} \sum_{j=1}^k \, y_j  \, x_{j,1}^{i_1} x_{j,2}^{i_2} \cdots x_{j,m}^{i_m} .  
 \end{equation}
The resulting natural parametrization of the conditional independence model is the map
\begin{equation}
\label{eq:tensorpara2}
 \Theta \,\rightarrow \, \Delta_n \,,\,\,\,
      (x,y) \,\mapsto \, \bigl(\,p_I(x,y) \,\bigr)_{I \in \Omega_{m,\ell}} . 
 \end{equation}
This polynomial map is $k! $-to-$1$, due to label swapping, which amounts to
permuting rows of $x$ and entries of $y$. The variety $X$ is the 
image in $\PP^n$ of the complexification of the map (\ref{eq:tensorpara2}).

Fix counts $s_I \in \NN$ for $I \in \Omega_{m,\ell}$.
Statisticians aim to maximize the log-likelihood function
$$ L \quad = \quad \sum_{I \in \Omega_{m,\ell}}  s_I \cdot {\rm log}\bigl( p_I(x,y) \bigr). $$
In this formula we incorporate the substitutions
$x_{j,m} = 1 -\sum_{i=1}^{m-1} x_{j,i}$ and $y_k = 1 - \sum_{j=1}^{k-1} y_j$.

We shall compute all complex critical points of $L$ by solving the likelihood equations
\begin{equation}
\label{eq:tensoreqns}
 \qquad  \frac{\partial L}{\partial x_{i,j}} \, = \, \frac{\partial L}{\partial y_i} \, = \, 0 \qquad
{\rm for}\, \, i =1,2,\ldots,k\, \,\,{\rm and}\, \,\, j = 1,2,\ldots,m-1. 
\end{equation}
This is a system of $km-1$ rational function equations in $km-1$ unknowns. We denote the corresponding rational map by $F(x,y,s): \CC^{km-1} \times \CC^{n+1} \dasharrow \CC^{km-1}$. 
The ML degree of the model $X$ is the number of complex solutions 
to the system (\ref{eq:tensoreqns}) divided by $k ! = 1\cdot 2 \,\cdots\, k $.
The maximum likelihood parameter  $(\hat x, \hat y)$ is one of the
real solutions in the polytope $\Theta$.

\begin{example}
Two small instances were studied in \cite{HKS}. The case $k=m=2,\ell=4$ 
is featured in \cite[Section 1]{HKS} 
where a gambler tosses one of two biased coins four times, and
$X \subset \PP^4$ is the hypersurface given by a $3 \times 3$ Hankel determinant.
This has ML degree~$12$, so (\ref{eq:tensoreqns}) has $24$ solutions.
A data vector $s$ with three local maxima in $\Delta_4$ is listed in
\cite[Example 10]{HKS}. In the table at the end of 
\cite[Section 5]{HKS} we learn that the model with $k=m=2,\ell=5$ has ML degree $39$,
so  (\ref{eq:tensoreqns}) has $78$ solutions. At that time, over $15$ years ago, 
symbolic computing with {\tt Singular} was the method of choice, and finding 
$78$ solutions  was not that easy.
\end{example}

Using the numerical methods presented in Section \ref{sec3}, we solved
 the likelihood equations for
$m=2,3$ and $k=2,3$. For various $\ell$, we ran many iterations
of the monodromy loop\footnote{ The optional argument {\tt group\_action} of {\tt monodromy\_solve} can be used to speed up the computations.} to count the number of solutions
to  (\ref{eq:tensoreqns}). Dividing that number by
$k!$ gives an integer, and that is the ML degree for the model.
A subsequent run of the certification feature in {\tt HomotopyContinuation.jl}
furnishes a proof that the proposed number is a lower bound on the ML degree.
However, our method does not give a proof that this is also an upper bound.

\begin{table}[h]
\setlength{\tabcolsep}{4pt}
\centering
\begin{tabular}{c|ccccccccc}
\backslashbox[8mm]{$k$}{$\ell$}
  & 4 & 5 & 6 & 7 & 8 & 9 & 10 \\ \hline
2 & 12 & 39 & 82 & 158 & 268 & 427 & 634 \\
3  & 1 & 1 & 111 & 645 &$\geq$ 2121 &  
\end{tabular}
\qquad \quad
\begin{tabular}{c|ccccccccc}
\backslashbox[8mm]{$k$}{$\ell$}
& 3 & 4 & 5  \\ \hline
2& 121 & 1449 & 8727 \\
3 & 646 & $\geq$ 100000 &
\end{tabular}
\caption{Experimentally obtained ML degrees for symmetric tensors of rank
 $k$ and order $\ell$. The size is
      $m = 2$ (left) or $m =3$ (right).
  Multiply by $k!$ for the number of solutions to~(\ref{eq:tensoreqns}).   }
\label{tab:tensors} 
\end{table}

\begin{remark}[A view from nonlinear algebra]
Points in the ambient space $\PP^n$ for our models
in Table \ref{tab:tensors} correspond to binary forms and ternary forms.
For example, the entry $111$ on the left is the ML degree for 
the $4 \times 4$ Hankel determinant  which defines binary sextics of rank~$3$.
The entry $646$ on the right concerns plane cubic curves of rank $3$.
This is the hypersurface in $\PP^9$ defined by the {\em Aronhold invariant},
shown in equation (9.15) and Example~11.12 in~\cite{MS}.
We solved the likelihood equations (\ref{eq:tensoreqns}) in the naive way,
by computing all $646 \times 3! = 3876$ zeros of the rational functions.
Further computational progress is surely possible. But,
just like in Example \ref{ex:RLM}, 
 this will require exploiting the special structure of the problem at hand.
 
 A next goal is the likelihood geometry of  $4 \times 4 \times 4 $ tensors.
  For a geometer, these are cubic surfaces in $\PP^3$, with parameters
   $m=4, \ell=3, n = 19$. 
  In the book cover~of~\cite{ASCB}, this means that
 DiaNA now juggles  three dice, each labeled ${\tt A},{\tt C},{\tt G},{\tt T}$. 
 We studied this model for cubic surfaces of rank $k =2$.
Our computations suggest  that the ML degree equals $6483$. 
\end{remark}

We next present an explicit numerical example, for
the model of plane cubics of rank $2$.

\begin{example}[$m=\ell=3,k=2,n=9$]
Consider the data vector $s \in \NN^{10}$ with coordinates
$$ \begin{small} \begin{matrix} 
s_{300}\, = \,8263, && s_{210}\, = \,4935, && s_{201}\, = \,8990, && s_{120}\, = \,7238, && s_{111}\, = \,5034, 
  \\ s_{102}\, = \,5106, &&  s_{030}\, = \,5181, && s_{021}\, = \,6843, && s_{012}\, = \,5282, && s_{003}\, = \,9501.
\end{matrix} \end{small} $$
The log-likelihood function $L$ has $242$ complex critical points, so there
are  $121$ critical points in the
secant variety $X \subset \PP^9$.  Precisely
eight of them lie in the actual model $X \cap \Delta_9 $. These come from
 $16$ critical points in 
$\Theta = \Delta_2 \times \Delta_2 \times \Delta_1$.
The maximum likelihood estimate equals
$$ \begin{small} \begin{matrix} 
\hat{p}_{300} \, = \,0.0661, & \hat{p}_{210} \, = \,0.1585, & \hat{p}_{201}\, = \,0.0937, & \hat{p}_{120}\, = \, 0.1269, & \hat{p}_{111}\, = \, 0.1542, &  \\ \hat{p}_{102}\, = \,0.0711, & \hat{p}_{030}\, = \,0.0340, & \hat{p}_{021}\, =0.0658\, , & \hat{p}_{012} \, = \,0.0883 , & \hat{p}_{003}\, = \, 0.1414.
\end{matrix} \end{small} $$
The $121$ critical points in $X$ are
 $3 {\times} 3 {\times} 3$~tensors of complex rank $2$.
Among these $121$ tensors, $47$ are real.
 We found that $20$ have real rank $2$, so each has two real preimages in $\RR^5$.
 The other $27$ have real rank $3$. They come from complex conjugate pairs
 of parameters $(x,y)$.
\end{example}

We now offer some pertinent remarks on
 numerical algebraic geometry. Our object of interest is the rational map
 $F: \CC^{km-1} \times \CC^{n+1} \dasharrow \CC^{km-1}$ defined by the gradient of~$L$.
 To find all solutions of $F(x,y;s^*) = 0$ for general complex data $ s^* \in \CC^{n+1}$,
  it is necessary that the monodromy action on $F(x,y;s^*)^{-1}(0)$ is transitive. This happens if and only if 
 the incidence variety $\overline{\{ (x,y,s): F(x,y,s) = 0 \} }$
 in the total space $ \CC^{km-1} \times \CC^{n+1}$ is irreducible \cite[Section~2]{duff2019solving}.
However, for our parametrized tensor models, this incidence variety is  reducible.

\begin{example}[$m=k=2, \ell = 4$] 
For any $s$ in $\CC^5$, we
consider  the solutions to the critical equations $ F  = \bigl( \frac{\partial L}{\partial x_{11}}, 
\frac{\partial L}{\partial x_{21}}, \frac{\partial L}{\partial y_1} \bigr) = 0 $
in the open subset where the denominators are nonzero.
The incidence variety $Y $ is the closure of this set in $ \CC^3 \times \CC^5$.
In a Gr\"obner basis approach, this would be computed
by clearing denominators in $F$ and then saturating the denominators.
To appreciate the complexity of this, note that the three numerators
have degree $25$, with  $2025$ terms,  $2418$ terms and $ 2439$ terms respectively.
This is why we do not clear denominators.
 
We see that $Y$ is reducible because
 all terms of $\frac{\partial L}{\partial x_{11}}$ are multiples of $y_1$.
Points in the locus $\{ y_1 = 0 \}$ parametrize  tensors of rank one. 
This gives an extraneous component of~$Y$. Interestingly, $Y$
has dimension $6$, because a rank $1$ tensor arises
from a line of parameter values, given by
$x_{11} = x_{21}$ and $y_1$ arbitrary.
The fibers of the map $Y \rightarrow \CC^5$ contain a
line and $ 24$ isolated points 
that represent $24/2! = 12$ rank-2 tensors. These are the critical points we are interested in. The corresponding 5-dimensional component of $Y$ parametrizes the
{\em likelihood correspondence}, i.e.~the irreducible variety in $\PP^4 \times \PP^4$ from
\cite[Definition 1.5]{HS}.  \end{example}

In summary, one drawback of our approach in this paper
is the presence of extraneous components in the incidence variety. 
From a numerical point of view, this makes the monodromy procedure more challenging. 
The phenomenon of \emph{path jumping} may bring us to other components, leading to the computation of spurious solutions. 
For computing the ML degree of our tensor models,
  we are only interested in critical points 
   in the regular locus of $X$. These are tensors of complex rank exactly $k$. They live on a
   component of $\,Y = \overline{\{F(x,y,s) = 0 \}} \subset \CC^{km-1} \times \CC^{n+1}$, called  the \emph{dominant component} in \cite[Remark~2.2]{duff2019solving}.
            We can compute all solutions on that component by making sure that our seed lies on it. 

\section{Positive Models and their Amplitudes} \label{sec6}

The physical theory of scattering amplitudes is concerned
with evaluating certain integrals of rational functions.
In our statistical setting, these correspond to {\em marginal likelihood integrals}
\begin{equation}
\label{eq:bayesintegral} \int_\Theta p_0(x)^{s_0} p_1(x)^{s_1} \cdots p_n^{s_n} \mu(x) {\rm d}x . 
\end{equation}
Such integrals arise in Bayesian statistics. In that paradigm
one integrates the likelihood function over the
parameter space $\Theta$ where the kernel is given by
a measure $\mu(x)$, known as the {\em prior belief}.
In general, it is a difficult problem to evaluate the integral
(\ref{eq:bayesintegral}) exactly and reliably.
See \cite{LS} for  an approach in the context of conditional independence as in Section \ref{sec5}.

It is a classical theme in mathematical statistics to connect
Bayesian inference with the optimization problem (MLE)
we explored in the previous sections. In this section we 
present new ideas for advancing that theme. These are inspired by 
positive geometries from Feynman diagrams and scattering amplitudes.
We build on the theory of stringy canonical forms~\cite{arkani2019stringy}.

\begin{definition} \label{def:positive}
A discrete statistical model $X$ is called  {\em positive}
if it has a parametrization  by positive rational functions $p_i(x)$
that sum to $1$, where the parameter space is the orthant $\Theta = \RR^d_{> 0}$.
 A positive rational function is the ratio of
 two polynomials with positive coefficients.
\end{definition}

Many familiar models in statistics are positive.
To begin with, the probability simplex $\Delta_n$ of all distributions
on $n+1$ states is a positive model,
thanks to the parametrization
\begin{equation} \label{eq:posimplex}  p : \RR_{>0}^n \,\rightarrow \, \Delta_n, \,\,\,
x \,\mapsto \, \frac{1}{1 {+} x_1{+} x_2 {+} \cdots {+} x_n} \bigl( 1,x_1,x_2,\ldots,x_n \bigr).  
\end{equation}
Next are the two families in \cite[Section 1.2]{ASCB}. In a 
 {\em toric model}, $p_i(x)$ is a monomial with a positive coefficient
divided by the sum of these $n+1$ monomials \cite{carlos, HS}. 
Every {\em linear model} $X$ is a positive model, since
$X \cap \Delta_n$ is a polytope whose vertices have nonnegative coordinates.
 For instance, a positive  parametrization $y \mapsto p(y)$ for Example~\ref{ex:sixeight} 
is found by replacing
\begin{equation}
\label{eq:x_to_y}
x_1 \,=\, \frac{y_1}{1+y_1+y_2+y_3}\,, \,\,
x_2 \,=\, \frac{y_1+y_2}{1+y_1+y_2+y_3} \,\,\,\, {\rm and} \,\,\,\,
x_3 \,=\, \frac{y_1+y_2+y_3}{1+y_1+y_2+y_3}. 
\end{equation}
Every model $X$ with ML degree one is a positive model,
by the parametrization in  \cite[Corollary 3.12]{HS}.
Mixtures of positive models are positive models, by
using (\ref{eq:posimplex}) for the mixture parameters.
In particular, all discrete conditional independence models
\cite[Chapter~4]{Sul} are positive models.
In the setting of Section~\ref{sec5}, 
we use (\ref{eq:posimplex}) to positively parametrize
the factors in $\,\Theta = (\Delta_{m-1})^k \times \Delta_{k-1}$,
and we then compose this with the positive polynomials in~(\ref{eq:tensorpara}).

Fix a positive model $X$. We factor the
numerator and denominator of each $p_i(x)$
into positive polynomials, we write
 $q_1(x), \ldots,q_e(x)$ for all the factors that occur, and
 we augment this list by $x_1,\ldots,x_d$.
We now rewrite the marginal likelihood integral (\ref{eq:bayesintegral}) 
in the form seen in \cite[(1.3)]{arkani2019stringy}.
To this end, we set $\varepsilon = 1$ and
 $\mu(x) = 1$ for now. Then the integral  (\ref{eq:bayesintegral}) becomes
\begin{equation}
\label{eq:bayesintegral2} \varepsilon^d \int_{\RR^d_{>0}}
\bigl[\,x_1^{u_1} \cdots x_d^{u_d} \,q_1(x)^{-v_1} q_2(x)^{-v_2} \cdots \, 
q_e(x)^{-v_e}\bigr]^\varepsilon  \,
\frac{{\rm d} x_1} {x_1} \cdots \frac{{\rm d} x_d}{x_d} ,
\end{equation}
where $u_i, v_j$ are certain $\ZZ$-linear combinations of $s_0,\ldots,s_n$.
The log-likelihood function equals
$$ L \,\,\, = \,\,\, \sum_{i=1}^d u_i \,{\rm log}(x_i) \,-\, \sum_{j=1}^e v_j \,{\rm log}(q_j(x)) . $$
The set ${\rm Crit}(L) \subset \CC^d$ of all critical points of $L$ can be computed
reliably using the methods  in this paper. We define
the {\em toric Hessian} of $L$ to be the symmetric $d \times d$-matrix
$H_L(x)$ whose entries are the rational functions $ \theta_i \theta_j L $, where
$\theta_i = x_i \partial_{x_i}$ is the $i$th Euler~operator.

In their recent work \cite{arkani2019stringy},
Arkani-Hamed, He and Lam  define the {\em string amplitude} of $L$ to be the
limit of the integral (\ref{eq:bayesintegral2}) as $\varepsilon$ tends to zero.
Given data $s$ such that all $v_j$ are positive, they consider the polytope
$P = \sum_{j=1}^e v_j\, {\rm New}(q_i)$ and assume
that $u = (u_1,\ldots,u_d)$ lies in $P$.

\begin{theorem} \label{thm:arkani}
The string amplitude of a positive model $X$ is a rational function
in the data $s_0,s_1,\ldots,s_n$. It equals the volume of the dual polytope
$(P- u)^*$, and it can be computed~as 
\begin{equation}
\label{eq:amplitude}
 {\rm amplitude}(X) \quad = \quad \sum_{\xi \in {\rm Crit}(L)} {{\rm det}(H_L(\xi))}^{-1}. 
 \end{equation}
\end{theorem}

\begin{proof}[Sketch of Proof]
This is our interpretation of the results in \cite{arkani2019stringy}.
The amplitude depends only on the
Newton polytopes ${\rm New}(q_j)$ and not on the
specific positive coefficients of $q_j(x)$.
The hypothesis that $u$ is in the interior of $P$ ensures that
 (\ref{eq:bayesintegral2}) converges
\cite[Section 4.1]{arkani2019stringy}. The volume formula
appears in \cite[(2.5)]{arkani2019stringy} for $s=1$
and in \cite[(4.15)]{arkani2019stringy} for $s\geq 2$.
The critical equations of $L$ are the saddle point equations for the 
marginal likelihood integral (\ref{eq:bayesintegral2}) when 
$\varepsilon \rightarrow \infty$. These equations appear in
\cite[Section 7.1]{arkani2019stringy}. They encode the
pushforward formula for canonical forms of positive geometries.
The toric Hessian is a convenient tool for writing the
Jacobian of the system \cite[(7.3)]{arkani2019stringy},
and hence for computing the integral in \cite[(7.5)]{arkani2019stringy}.
\end{proof}

\begin{example}[$d=n=2$] \label{ex:triangle}
For the model $X = \PP^2$,
parametrized by~(\ref{eq:posimplex}), the integral (\ref{eq:bayesintegral2}) is
$$ \varepsilon^2 \int_0^\infty \int_0^\infty \left[ \frac{x_1^{s_1} x_2^{s_2}}{(1+x_1+x_2)^{s_0+s_1+s_2} }\right]^\varepsilon
\frac{{\rm d}x_1}{x_1} \frac{{\rm d}x_2}{x_2}  .$$
The log-likelihood function
$L = s_1 {\rm log}(x_1) + s_2 {\rm log}(x_2) - (s_0 + s_1 + s_2) {\rm log}(1 + x_1 + x_2)$
has only one critical point, namely $(\hat x_1 ,\hat x_2) = \frac{1}{s_0}( s_1, s_2)$.
Substituting this into $1/{\rm det}(H_L(x))$, we~get
$$ {\rm amplitude}(X) \quad = \quad \frac{1}{s_0 s_1} + \frac{1}{s_0 s_2} + \frac{1}{s_1 s_2} \quad = \quad
\quad {\rm area}\bigl(\,(\, P - (s_1,s_2)\,)^* \,\bigr) . $$
Here, $P$ is the unit triangle ${\rm conv}\{ (0,0),(0,1),(1,0) \}$
scaled by the sample size $s_0+s_1+s_2$.
\end{example}
 
 The special case $e=1$  in Theorem~\ref{thm:arkani} corresponds to
 the class of {\em toric models} in statistics; see \cite[Section~3]{HS} and \cite[Section~1.2]{ASCB}.
 Any polynomial 
$q(x) = \sum_{j=0}^n c_j x^{{\bf a}_j}$
with positive coefficients $c_j > 0$
defines a toric model $X$, by setting
$p_j(x) = c_j x^{{\bf a}_j}/q(x)$ for $j=0,\ldots,n$.
The ML degree of $X$ depends in subtle ways on the
coefficients $c_j$. This was observed in \cite[Section~7.1]{arkani2019stringy} and
studied in detail in \cite{carlos}. Both sources contain many open
problems. For instance, it is conjectured in \cite{arkani2019stringy}
that the number $(m-3)!$ from Section~\ref{sec2} is the
minimal ML degree among all toric models 
supported on the associahedron.
The diffeomorphism referred to in \cite[Claim 4]{arkani2019stringy}
is the familiar toric {\em moment map} \cite[Theorem 8.24]{MS}.
The amplitude of the toric model $X$ equals the 
{\em adjoint} of the dual Newton polytope $P^*$,
in the sense of Wachspress geometry \cite{KR},
after dividing by the product of
the linear forms given by the facets of $P^*$.
We learned this from unpublished lecture notes by Christian Gaetz
which connect \cite{arkani2017positive} with~\cite{KR}.

\begin{example}[Measuring the dual of a square]
The toric model for $q(x) = 1+x_1+x_2 + x_1 x_2$ is
the {\em independence model for two binary random variables},
with data $s = (s_{ij})_{0 \leq i,j \leq 1}$.
Here $n=3,d=2$, and $X$ is the Segre quadric in $\PP^3$.
The marginal likelihood integral in (\ref{eq:bayesintegral2}) is
$$ \varepsilon^2 \int_0^\infty \int_0^\infty \left[ \frac{x_1^{s_{10}+s_{11}} x_2^{s_{01}+s_{11}}}
{((1+x_1)(1+x_2))^{s_{00}+s_{01}+s_{10}+s_{11}}}\right]^\varepsilon
\frac{{\rm d}x_1}{x_1} \frac{{\rm d}x_2}{x_2}  .$$
The limit for $\varepsilon \rightarrow 0$ is the string amplitude. 
Its denominator is the product of the row and column sums of the contingency table $s$.
The adjoint is the square of the sample size. Hence,
$$ {\rm amplitude}(X) \,\, = \,\, \, \frac{(s_{00} + s_{01} + s_{10} + s_{11})^2}{
(s_{00}+s_{01})(s_{10}+s_{11}) ( s_{00}+s_{10})(s_{01}+s_{11})} .$$
Here $P$ is the square $[0,1]^2$ times the sample size. This is  translated
  by $u= (s_{10}+s_{11}, s_{01}+s_{11})$.
The normalized area  of the dual quadrilateral $(P-u)^*$ equals
the  string amplitude. Note that the assumption $u \in P$ 
from Theorem \ref{thm:arkani} is naturally satisfied in the statistical setting. 
\end{example}

In earlier sections we showed that {\tt HomotopyContinuation.jl} is fast for computing 
the critical set ${\rm Crit}(L)$  of the log-likelihood function $L$. 
And it comes with certification.
We use this to compute the 
sum (\ref{eq:amplitude})
and hence to evaluate string amplitudes for positive models.
While the meaning of these amplitudes for Bayesian statistics is not clear yet,
there is considerable interest in such computations among particle physicists.
We next illustrate this for  the
CHY and CEGM models in Sections~\ref{sec2} and \ref{sec4}.
We follow the set-up in \cite[Section 6.2]{arkani2019stringy}.

Let us begin with  the $k=2$ model in Section \ref{sec3},
with positive reparametrization as in~(\ref{eq:x_to_y}).

\begin{example}[$k=2,~m=6$] \label{ex:amplitudek2m6}
We compute the amplitude for the CHY model in
Example~\ref{ex:sixeight}.
In terms of the positive parameters $y_1,y_2,y_3$ from \eqref{eq:x_to_y},
the log-likelihood function in  (\ref{eq:loglike62})~is
$$ \begin{matrix}
L & =&\!\!  s_{23} \log(y_1)  + s_{34} \log(y_2) + s_{45} \log(y_3) + s_{24} \log(y_1 {+} y_2) + s_{25} 
\log(y_1 {+} y_2 {+} y_3) \,+ \qquad \qquad \\
& &\,\, s_{35} \log(y_2 {+} y_3) + s_{36} \log(1 {+} y_2 {+} y_3) + s_{46} \log(1{+}y_3) 
\,-\,  (\sum_{(i,j)} s_{ij} ) \log(1 {+} y_1 {+} y_2 {+} y_3) .
\end{matrix}
$$
The toric Hessian $H_L(y)$ is a symmetric $3 \times 3$-matrix
whose entries are rational functions.
The sum of the values of $- {\rm det}(H_L(y))^{-1}$ at the six critical points of $L$ is the string amplitude
\begin{equation}
\label{eq:14terms}  \begin{matrix}
& \,\,\frac{1}{s_{12} s_{34} s_{56}}
+\frac{1}{s_{12} s_{56}  s_{123}}
+\frac{1}{s_{23} s_{56}  s_{123}}
+\frac{1}{s_{23} s_{56}  s_{234}}
+\frac{1} {s_{34} s_{56} s_{234}}
+ \frac{1}{s_{16} s_{23} s_{45}}
+ \frac{1}{ s_{12} s_{34} s_{345}} \smallskip  \\
&   +\, \frac{1}{s_{12} s_{45} s_{123}}
+ \frac{1}{s_{12} s_{45} s_{345}}
+\frac{1}{s_{16} s_{23} s_{234}}
+\frac{1} {s_{16} s_{34} s_{234}}
+\frac{1}{s_{16} s_{34} s_{345}}
+\frac{1}{s_{16} s_{45} s_{345}}
+\frac{1}{s_{23} s_{45} s_{123}}.
\end{matrix}
\end{equation}
Here we abbreviate  $s_{ijk} = s_{ij} + s_{ik} + s_{jk}$. The $14$ terms in this sum correspond to
the planar trivalent trees with six labeled leaves, and hence to 
 the vertices of the {\em associahedron} in $\RR^3$.
 
 For a numerical example take the data  in (\ref{eq:data62}) and (\ref{eq:kinemat}).
The unique positive critical point  $(\hat y_1, \hat y_2, \hat y_3) =  (1.076..., 1.202...,  1.205...)$
maps to the MLE in (\ref{eq:learned}). The amplitude (\ref{eq:14terms})~equals
{\begin{small} $$
\frac{16074421}{56770632000} \,\, = \,\,  0.00028314676856...
$$
\end{small} }
Using the abbreviation $y_{i,j} = \sum_{i \leq \ell \leq j} y_\ell$, the associated integral
(\ref{eq:bayesintegral2}) has the form
$$ \varepsilon^3 \int_{\RR^3_{>0}} \left[ \frac{y_1^{s_{23}} y_2^{s_{34}} y_3^{s_{45}}}
{y_{1,2}^{-s_{24}}y_{1,3}^{-s_{25}}y_{2,3}^{-s_{35}}(1 + y_{2,3})^{-s_{36}}(1+y_3)^{-s_{46}}(1 + y_{1,3})^{\sum_{(i,j)} s_{ij}}}\right]^\varepsilon
\frac{{\rm d}y_1}{y_1} \frac{{\rm d}y_2}{y_2} \frac{{\rm d}y_3}{y_3}   .$$
The theory in \cite{arkani2019stringy} requires the hypotheses
$\, s_{23} \geq 0,~s_{34} \geq 0,~s_{45} \geq 0,~s_{24} \leq 0, ~ s_{25} \leq 0, ~ s_{35} \leq 0, ~ s_{36} \leq 0, ~ s_{46} \leq 0, ~ \sum_{(i,j)} s_{i,j} \geq 0$.  If this holds then
the leading order ($\varepsilon \rightarrow 0$) of the integral equals the volume of $(P-(s_{23},s_{34},s_{45}))^*$,
where $P$ is the associahedron
\begin{small} $$    c_{24} {\rm New}(y_{1,2}) + c_{25}{\rm New}(y_{1,3}) 
+ c_{35} {\rm New}(y_{2,3}) + c_{36} {\rm New}(1+y_{2,3}) 
+ c_{46} {\rm New}(1+y_3) + \sum_{(i,j)} s_{ij} {\rm New}(1+y_{1,3}). 
$$
\end{small} 
Here $c_{ij} = -s_{ij}$.
The hypothesis fails for (\ref{eq:data62}), but
summing over ${\rm Crit}(L)$ always works.
\end{example}

We now reiterate the punchline from Section~\ref{sec2} for amplitudes:~{\em current off-the-shelf software from numerical algebraic geometry is
  highly efficient and reliable for computing string amplitudes by
evaluating  the sum (\ref{eq:amplitude}).   For our computations we
   used  {\tt HomotopyContinuation.jl}} \cite{BRT, BT}. 
   We carried this out  for models with   $k =2$ and $k=3$.
If $v_1,\ldots,v_e > 0$ then
the amplitude measures the
volume of the dual polytope in Theorem~\ref{thm:arkani}.

For $k=2$, our computations validate known
formulas involving planar trees like (\ref{eq:14terms}).
For $k = 3, m \leq 7$, Cachazo et al.~\cite{AG} 
describe formulas
in terms of rays of the positive tropical Grassmannian,
but in general there is still plenty of room for further discovery.

\begin{example}[$k = 2$]
We used the positive parametrization \cite[(1.5)]{arkani2019stringy} of $\mathcal{M}^+_{0,m}$ to verify \eqref{eq:amplitude} for the CHY model.
   Fix integer values for the Mandelstam invariants such that the hypotheses on $u$ and $v$ in Theorem \ref{thm:arkani} are satisfied. We compute the volume of $(P-u)^*$ in two ways. First the exact rational number is obtained using {\tt Polymake.jl} \cite{KLT}. Secondly, summing over the computed critical points as in \eqref{eq:amplitude} gives a floating point approximation.  The cases we checked are $m = 5, 6, \ldots, 10$. Using double precision arithmetic, the numerical evaluation of \eqref{eq:amplitude} agrees with the volume up to at least 12 significant digits in all  cases. Computing the Hessian determinant and summing over the 5040 solutions for $m = 10$ takes about 20 seconds. The computation time for finding these solutions appears in Table \ref{tab:scatteringk2}.
\end{example}

\begin{example}[$k=3$]
For $m = 7$, we compute the string amplitude of the CEGM model for the
 data in (\ref{eq:data37}). Our code finds the numerical value
$ 3.5930250842 \cdot 10^{-19}$. 
This equals
\begin{scriptsize}
$$ \frac{33816298756441110112144384689994772682422222303915555493151512849490472904959110279}
{94116511751278934751720147762872529781445322700567349835265330333349828096628760994174245472501760000}.$$ \end{scriptsize}  
This rational number is computed with a formula from \cite[Section 4]{AG} which was kindly shared with us by
Nick Early. In our study of the string amplitudes for CEGM models,
  we used the positive parametrization obtained 
  from \eqref{eq:threebym} by recursively setting $x_0 = y_0 = 1$ and
\begin{equation} \label{eq:posparam}
x_\ell = x_{\ell-1} + z_\ell, \qquad y_\ell = y_{\ell-1} + z_\ell(1+w_1 + \cdots + w_\ell), \qquad \ell	= 1, \ldots m-4
\end{equation}
 Since this parametrization augments the degree of the equations, it is better 
  to first solve the scattering equations using the formulation \eqref{eq:threebym} and then compute the $(z,w)$ coordinates of the solutions $(x,y)$ via \eqref{eq:posparam}. Computing the sum \eqref{eq:amplitude} over the 1272 solutions takes about 11 seconds. Like Example \ref{ex:amplitudek2m6}, this illustrates the validity of \eqref{eq:amplitude} when the assumptions on $u,v$ in Theorem \ref{thm:arkani} are violated. For $m = 8$, we obtain the numerical approximation $1.3609103649662523 \cdot 10^{-34}$ for the string amplitude of the CEGM model with data \eqref{eq:bigmandelstam}.
\end{example}

We conclude with a summary of
what has been accomplished in this paper.
A connection has been made between
algebraic statistics and the study of
scattering amplitudes in physics.
Positive models play the role of positive geometries.
We showed how to solve the
likelihood equations with
certified numerical methods,
and how to use this for evaluating amplitudes.
Our case study offers a new tool kit for
 statistics and  physics, based on nonlinear algebra.

Here is what we did not do:
we did not prove new theorems in pure mathematics.
We did not achieve notable methodological progress
in statistics or theoretical advances in physics.
The contribution of this work lies in building a bridge.
Others may now cross that bridge, and  use our tool kit to gain
insights on the numerous fascinating problems that remain open.

\bigskip

\noindent
{\bf Acknowledgement}.
We are very grateful to Sascha Timme for his help with the
software {\tt HomotopyContinuation.jl},
and to Pieter Bomans and Taylor Brysiewicz for discussions. 
We thank Freddy Cachazo and Nick Early for 
patient tutoring and inspiring conversations.

\bigskip \bigskip

\noindent
\footnotesize 
{\bf Authors' addresses:}

\smallskip

\noindent Bernd Sturmfels,
MPI-MiS Leipzig and UC Berkeley 
\hfill {\tt bernd@mis.mpg.de}

\noindent Simon Telen, MPI-MiS Leipzig
\hfill {\tt simon.telen@mis.mpg.de}


\bigskip

\begin{thebibliography}{10}
\begin{small}
\setlength{\itemsep}{-0.5mm}

\bibitem{carlos}
C.~Am\'endola, N.~Bliss, I.~Burke, C.~Gibbons, M.~Helmer, S.~Ho\c sten, E.~Nash, J.~Rodriguez
and D.~Smolkin:
{\em The maximum likelihood degree of toric varieties}, J. Symbolic Computation {\bf 92} (2019), 222--242.

\bibitem{arkani2018scattering}
N.~Arkani-Hamed, Y.~Bai, S.~He and G.~Yan:
{\em Scattering forms and the positive geometry of kinematics, color and
  the worldsheet}, Journal of High Energy Physics (2018), no 5, 095.

\bibitem{arkani2017positive}
N.~Arkani-Hamed, Y.~Bai and T.~Lam:
{\em Positive geometries and canonical forms}, 
Journal of High Energy Physics (2017), no 11, 039.

\bibitem{NAH}
N.~Arkani-Hamed, J.~Bourjaily, F.~Cachazo, A.~Goncharov, A.~Postnikov and J.~Trnka:
{\em Grassmannian Geometry of Scattering Amplitudes}, Cambridge University Press, 2016.

\bibitem{arkani2019stringy}
N.~Arkani-Hamed, S.~He and T.~Lam:
{\em Stringy canonical forms}, {\tt arXiv:1912.08707}.

\bibitem{arkani2020positive}
N.~Arkani-Hamed, T.~Lam and M.~Spradlin:
{\em Positive configuration space}, {\tt arXiv:2003.03904}.

\bibitem{BRT}
P.~Breiding, K.~Rose and S.~Timme:
{\em Certifying zeros of polynomial systems using interval arithmetic}, {\tt arXiv:2011.05000}.

\bibitem{BT}
P.~Breiding and S.~Timme:
{\em HomotopyContinuation.jl: A package for homotopy continuation in julia},
International Congress on Mathematical Software, 458--465, Springer, 2018.

\bibitem{CUZ}
F.~Cachazo, B.~Umbert and Y.~Zhang:
{\em Singular solutions in soft limits},
Journal of High Energy Physics (2020), no 5, 148.

\bibitem{AG}
F.~Cachazo, N.~Early, A.~Guevara and S.~Mizera: {\em Scattering equations: from 
projective spaces to tropical Grassmannians},  Journal of High Energy Physics (2019), no 6, 039.

\bibitem{cachazo2014scattering}
F.~Cachazo, S.~He and E.~Y. Yuan:
{\em Scattering equations and Kawai-Lewellen-Tye orthogonality},
Physical Review D {\bf 90} (2014) 065001.

\bibitem{CHKS} 
 F.~Catanese, S.~Ho\c sten, A.~Khetan and B.~Sturmfels:
{\em The maximum likelihood degree},  American Journal of Mathematics {\bf 128} (2006) 671--697. 

\bibitem{CD}
E.~Cattani and A.~Dickenstein:
{\em A global view of residues in the torus},
J.~Pure Appl.~Algebra {\bf 117} (1997) 119-144.

\bibitem{duff2019solving}
T.~Duff, C.~Hill, A.~Jensen, K.~Lee, A.~Leykin, and J.~Sommars:
{\em Solving polynomial systems via homotopy continuation and monodromy},
IMA J.~Numerical Analysis {\bf 39} (2019) 1421--1446.

\bibitem{HKS}
S.~Ho\c sten, A.~Khetan and B.~Sturmfels:
{\em Solving the likelihood equations}, Foundations of Computational Mathematics
{\bf 5} (2005) 389--407. 

\bibitem{HS}
J.~Huh and B.~Sturmfels:
{\em Likelihood geometry},  Combinatorial Algebraic Geometry (eds. Aldo Conca et al.), 
Lecture Notes in Mathematics 2108, Springer, (2014) 63--117. 

\bibitem{KLT}
M.~Kaluba, B.~Lorenz and S.~Timme.
{\em Polymake.jl: A new interface to polymake},
Mathematical Software  -- ICMS 2020, Springer Lecture Notes in Computer Science,  vol 12097, 377-385, 2020.

\bibitem{KR}
K.~Kohn and K.~Ranestad:
{\em Projective geometry of Wachspress coordinates},
Foundations of Computational Mathematics~{\bf 20} (2020) 1135-1173.

\bibitem{LS}
S.~Lin and B.~Sturmfels:
{\em Marginal likelihood integrals for mixtures of independence models},
Journal of Machine Learning Research {\bf 10} (1009) 1611-1631.

\bibitem{KV} M.~Kummer and C.~Vinzant: 
{\em The Chow form of a reciprocal linear space},
Michigan Mathematical Journal {\bf 68} (2019) 831--858.

\bibitem{MS} M.~Micha\l ek and B.~Sturmfels:
{\em Invitation to Nonlinear Algebra},
Graduate Studies in Mathematics, vol 211, American Mathematical Society, 2021.


\bibitem{ASCB}
L.~Pachter and B.~Sturmfels:
{\em Algebraic Statistics for Computational Biology},
Cambridge University Press, 2005.

\bibitem{RW}
J.~Rodriguez and B.~Wang: {\em Computing Euler obstruction functions 
using maximum likelihood degrees}, Int.~Math.~Res.~Not.~IMRN {\bf 20} (2020) 6699--6712. 

\bibitem{sommese2005numerical}
A.~Sommese and C.~Wampler:
{\em The Numerical Solution of Systems of Polynomials Arising in
  Engineering and Science},
World Scientific Publishing, Hackensack, 2005.

\bibitem{Sul}
S.~Sullivant: {\em Algebraic Statistics},
Graduate Studies in Mathematics, 194, American Mathematical Society, Providence, RI, 2018.
 
\bibitem{TW}
K.~Talaska and L.~Williams:
{\em Network parametrizations for the Grassmannian},
 Algebra and Number Theory  {\bf 7} (2013) 2275--2311.

 
\end{small}
\end{thebibliography}
\end{document}